\documentclass[10pt]{article} 

\usepackage{amsmath} 
\usepackage[all,cmtip]{xy}
\usepackage{graphicx}
\usepackage{color,verbatim}
\usepackage{hyperref}
\usepackage{amsthm}
\usepackage{amssymb}
\usepackage{mathrsfs,mathtools}
\usepackage{authblk,array}

\theoremstyle{plain}
\newtheorem{theorem}{Theorem}[section]
\newtheorem{proposition}[theorem]{Proposition}
\newtheorem{lemma}[theorem]{Lemma}
\newtheorem{corollary}[theorem]{Corollary}

\theoremstyle{definition}
\newtheorem{definition}[theorem]{Definition}
\newtheorem{example}[theorem]{Example}
\newtheorem{remark}[theorem]{Remark}
\newtheorem*{con}{Convention}

\DeclareMathOperator{\colim}{colim}
\DeclareMathOperator{\wayb}{{\rotatebox[origin=c]{-90}{$\twoheadrightarrow$}}}

\newcommand{\oto}{{\longrightarrow\hspace*{-3.1ex}{\circ}\hspace*{1.7ex}}}  

\input diagxy

\usepackage[utf8]{inputenc} 

\usepackage{geometry} 
\title{The bounded ideal monad on the category of quasi-metric spaces and its algebras}

\author{
	Kai Wang$^1$,  Dexue Zhang$^2$ \\
	\small  $^1${\thinspace}School of Mathematics and Information Sciences, Yantai University, Yantai, China \\
	\small   $^2${\thinspace}School of Mathematics, Sichuan University, Chengdu, China \\
	\small Email: wangkai5111@126.com (K. Wang), dxzhang@scu.edu.cn (D. Zhang)
}

\date{}  

\begin{document} 
	\maketitle
	
	\begin{abstract}
		The notion of bounded ideals is introduced for quasi-metric spaces. Such ideals give rise to a monad, the bounded ideal monad, on the category of quasi-metric spaces and non-expansive maps. Algebras of this monad are metric version of local dcpos of Mislove. It is shown that an algebra of the bounded ideal monad is a standard quasi-metric space of which the formal balls form a local dcpo; and that a continuous algebra is a standard quasi-metric space of which the formal balls form a local domain.

		
		
	\end{abstract}

	\section{Introduction} 
	From the viewpoint of category theory, quasi-metric spaces are categories enriched over the closed category $([0,\infty],\geq,+,0)$. In his influential paper \cite{Lawvere1973}, Lawvere  stressed the importance of this fact and observed that Cauchy completeness of a metric space (i.e., every Cauchy sequence converges)  can be formulated as a categorical property (i.e., every Cauchy weight is representable). This is an inspiring example of that a property of quasi-metric spaces may have two facets: a domain theoretic one (in terms of topology and convergence) and a category theoretic one. Besides Cauchy completeness, nice examples in this regard include Yoneda completeness and Smyth completeness:   
	\begin{itemize}
		\item (Flagg,   S\"{u}nderhauf and  Wagner \cite[Lemma 46]{Flagg1996}) A quasi-metric space is Yoneda complete in the sense that every forward Cauchy net has a Yoneda limit if, and only if, all of its ideals (which are sort of metric version of ind-objects of an \emph{ordinary} category) have a colimit. 
		\item (Li and Zhang \cite[Proposition 6.3]{LiZ2018a}) A quasi-metric space is Smyth complete in the sense that every forward Cauchy net converges in the open ball topology of its symmetrization \cite{Sunderhauf1995,KS2002}  if, and only if, all of its ideals are representable.
	\end{itemize}
	
	A key feature of quasi-metric spaces is that, as summarized  in the motto of \cite{Goubault2017}, ``Formal balls are the essence of quasi-metric spaces''. That means, roughly, that properties of quasi-metric spaces are reflected by the order theoretic structure of their formal balls. Here are some examples:  \begin{itemize}
		\item (Kostanek and Waszkiewicz \cite[Theorem 7.1]{Kostanek2010}) A quasi-metric space $(X,d)$ is Yoneda complete if and only if its set of formal balls is directed complete. 
		\item (Romaguera and Valero \cite[Theorem 3.2]{Romaguera-V2010}) A quasi-metric space $(X,d)$ is Smyth complete if and only if its set of formal balls is a continuous dcpo with way below relation given by $(x,r)\ll(y,s)\Leftrightarrow r>s+d(x,y)$.
		\item (Goubault-Larrecq and Ng \cite[Theorem 3.7]{Goubault2017}) A quasi-metric space $(X,d)$ is Yoneda complete and continuous in the sense of \cite{Kostanek2010,Waszkiewicz2009} if and only if its set of formal balls is a continuous dcpo.
	\end{itemize} 
	
	It is known that assigning to each partially ordered set the ordered set of its ideals (= directed lower sets) defines a monad on the category {\sf Poset} of partially ordered sets and order-preserving maps, of which the category of algebras is directed complete partially ordered sets (dcpos for short) and Scott continuous maps (i.e., maps preserving directed joins). It is not hard to check that bounded ideals (i.e., ideals with an upper bound) also give rise to a monad on the category {\sf Poset}, of which the category of algebras is the local dcpos of Mislove \cite{Mislove1999} and Scott continuous maps. 
	
	Yoneda complete quasi-metric spaces are metric version of dcpos, they are the algebras of the ideal monad on the category {\sf qMet} of quasi-metric spaces and non-expansive maps. In this paper we introduce the notion of bounded ideals for quasi-metric spaces. Such ideals give rise to a monad -- the bounded ideal monad -- on the category {\sf qMet}. Algebras of the bounded ideal monad are metric version of local dcpos. Characterizations of such algebras via their formal balls are presented,  echoing the motto of \cite{Goubault2017}. First, a quasi-metric space  is an algebra of the bounded ideal monad if and only if it is standard in the sense of Goubault-Larrecq and Ng \cite{Goubault2017} and its set of formal balls is a local dcpo, hence local Yoneda complete in the sense of Lu and Zhao \cite{Lu2023}. Second,  a quasi-metric space  is a continuous algebra of the bounded ideal monad if and only if it is standard and its set of formal balls is a local domain. 
	
	\section{Yoneda limits and formal balls} 
	
	In this section we recall some basic notions about quasi-metric spaces, the only new result is an equivalent reformulation of the notion of standard quasi-metric spaces introduced in \cite{Goubault2017}.  Our references to quasi-metric spaces are \cite{Bon-Rutten1997} and \cite{Goubault2013}, to order theory are  \cite{Gierz-Hofmann-Lawson2003} and \cite{Goubault2013}.
	
	A quasi-metric space is a pair $(X,d)$ consisting of a set $X$ and a map $d\colon X\times X\longrightarrow[0,\infty]$  such that for all $x,y,z\in X$: \begin{itemize}  \item $d(x,x)=0$; \item $d(x,y)+d(y,z)\geq d(x,z)$; \item   $d(x,y)=d(y,x)=0\Rightarrow x=y$.    \end{itemize}   
	
	A classical metric is just a symmetric (i.e., $d(x,y)=d(y,x)$) and finitary (i.e., $d(x,y)<\infty$) quasi-metric.
	
	The \emph{underlying order} (or \emph{specialization order}) $\sqsubseteq$ of a quasi-metric space $(X,d)$ refers to  the order relation on $X$ defined by  $x\sqsubseteq y$ if $d(x,y)=0.$  
	
	\begin{con}
		We reserve the symbols $\leq,<,\geq,>$ and $\inf,\sup$ for the usual order  of real numbers, and use $\sqsubseteq$ to denote other order relations. An exception  is that the order relation of the index set  of a net will also be denoted by $\leq$, for example, we'll write $\inf_{i\in D} \sup_{k\geq j\geq i}d(x_j,x_k)=0$.
	\end{con}  
	
	\begin{example} \label{internal hom}
		For all $a,b\in [0,\infty]$, let
		$$d_L(a,b)=b\ominus a, \quad  d_R(a,b)=a\ominus b,$$ where $\ominus$ is the truncated minus, which is to say, $a\ominus b\coloneqq \max\{0,a-b\}$. 
		Then  $d_L$ and $d_R$ are quasi-metrics  on $[0,\infty]$.  These two quasi-metrics are opposite to each other in the sense that $d_R(a,b)=d_L^{\rm op}(a,b)\coloneqq d_L(b,a)$. The underlying order of $([0,\infty],d_L)$ is the usual relation $\geq$ between real numbers; the   underlying order of $([0,\infty],d_R)$ is the usual relation $\leq$ between real numbers. 
	\end{example}

	Let $f\colon (X,d_X)\longrightarrow (Y,d_Y)$ be a map between quasi-metric spaces. We say that \begin{itemize}\item $f$ is \emph{non-expansive} if  for all $x,y\in X$,
		$ d_X(x,y)\geq d_Y(f(x),f(y)) $;   \item
		$f$   is \emph{isometric} if  for all $x,y\in X$,  $d_X(x,y)= d_Y(f(x),f(y)) $. \end{itemize}

	Quasi-metric spaces and non-expansive maps form a category $ \sf qMet$. Since any non-expansive map preserves the underlying order, taking underlying order defines a functor from the category $ \sf qMet$ to the category of partially ordered sets and maps preserving order. For any quasi-metric spaces $(X,d_X)$ and $(Y,d_Y)$, the relation $\sqsubseteq$ on the set of  non-expansive maps from  $(X,d_X)$ to $(Y,d_Y)$ given by $$f\sqsubseteq g \quad\text{if}\quad \forall x\in X,~ f(x)\sqsubseteq_Y g(x)$$ is a partial order. With this order relation on the hom-sets the category $\sf qMet$ becomes a locally ordered category, hence a  2-category.
	
	The category $\sf qMet$ is complete and cocomplete. In particular, for every set $X$  the $X$-power of $([0,\infty],d_L)$ is given by $([0,\infty]^X,\rho_X)$, where
	$$  \rho_X(\lambda,\mu)=\sup_{x\in X}\mu(x)\ominus \lambda(x)$$ for all $\lambda,\mu\in [0,\infty]^X$. The underlying order $\sqsubseteq$ of  $([0,\infty]^X,\rho_X)$ is given by \[\lambda\sqsubseteq \mu\quad \text{if and only if}\quad  \forall x\in X,\lambda(x)\geq \mu(x).\]  
	
	Suppose $(X,d)$ is a quasi-metric space; suppose $\{x_i\}_{i\in D} $ is a net and $x$ is an element  of   $(X,d)$. We say that 
	\begin{itemize}\item 	
		$\{x_i\}_{i\in D} $ is forward Cauchy if 
		$$\inf_{i\in D} \sup_{k\geq j\geq i} d(x_j,x_k)=0.$$
		\item  $x$ is a Yoneda limit of $\{x_i\}_{i\in D}$ if for all $y\in X$, 
		$$d(x,y)=\inf_{i\in D} \sup_{ j\geq i} d(x_j,y).$$ \end{itemize}

	A quasi-metric space is  \emph{Yoneda complete}  if each of its forward Cauchy nets has a Yoneda limit. A non-expansive map $f\colon (X,d_X)\longrightarrow (Y,d_Y)$ is \emph{Yoneda continuous}  provided that it preserves Yoneda limits of forward Cauchy nets in the sense for each forward Cauchy net $\{x_i\}_{i\in D} $ of $(X,d)$, if $x$ is a Yoneda limit of $\{x_i\}_{i\in D} $, then $f(x)$ is a Yoneda limit of $\{f(x_i)\}_{i\in D} $. The notion of Yoneda limits of forward Cauchy nets originated in Smyth  \cite{Smyth88}. For more information on Yoneda complete quasi-metric spaces the reader is referred to \cite{Bon-Rutten1997,Goubault2013,Goubault2017,KS2002}.

	A formal ball of a  quasi-metric space $(X,d)$ is a pair $(x,r)$ with $x\in X$ and $r\in [0,\infty)$. The relation $\sqsubseteq$ on the set of formal balls, given by
	$$(x,r)\sqsubseteq(y,s)\quad \text{if}\quad  d(x,y)\leq r-s,$$ 
	is reflexive, transitive, and anti-symmetric, hence a partial order. We write $(\mathrm{B}X,\sqsubseteq)$, or simply  $\mathrm{B}X$, for the set of formal balls of $(X,d)$ ordered by $\sqsubseteq$. The formal ball construction dates back to Weihrauch and Schreiber \cite {Weihrauch1981}, it is a bridge between quasi-metric spaces and domain theory, see e.g. \cite{Edalat1998,Goubault2013,Goubault2019,Goubault2017}.
	
	In this paper, each directed set of $\mathrm{B}X$ will be viewed as a monotone net indexed by itself. That means, we'll write a directed set of $\mathrm{B}X$ as a net $\{(x_i,r_i)\}_{i\in D}$ with $$i\leq j\quad \text{if and only if}\quad   (x_i,r_i)\sqsubseteq(x_j,r_j).$$ 
	
	Yoneda limits are closely related to joins of directed sets of $\mathrm{B}X $. If $\{(x_i,r_i)\}_{i\in D}$ is a directed set of $\mathrm{B}X $,  the net $\{x_i\}_{i\in D}$ is  clearly forward Cauchy.  The following lemma, contained in  \cite[Theorem 3.3]{Ali-Akbari2009}, also in   \cite[Lemma 7.7]{Kostanek2010} and  \cite[Lemma 7.4.25]{Goubault2013}, says that if the forward Cauchy net $\{x_i\}_{i\in D}$ has a Yoneda limit, then the directed set $\{(x_i,r_i)\}_{i\in D}$ has a join.
	
	\begin{lemma}\label{join of directed set in BX}  
		Suppose $(X,d)$ is a quasi-metric space and   $\{(x_i,r_i)\}_{i\in D}$ is a directed set of $\mathrm{B}X $. If  $x$ is a Yoneda limit of the forward Cauchy net $\{x_i\}_{i\in D}$  and  $r=\inf_{i\in D}r_i$,  then $(x,r)$ is a join of  $\{(x_i,r_i)\}_{i\in D}$ in $\mathrm{B}X $. \end{lemma}
	
	The converse of Lemma \ref{join of directed set in BX} is not true in general. This leads to the following
	
	\begin{definition} \label{defn of standard} A quasi-metric space $(X,d)$ is standard provided that whenever a directed set  $\{(x_i,r_i)\}_{i\in D}$ of formal balls has a join,   the forward Cauchy net $\{x_i\}_{i\in D}$ has a Yoneda limit. \end{definition}
	
	In other words, a quasi-metric space $(X,d)$ is standard if the joins of directed sets of $\mathrm{B}X $ are computed componentwise in the sense that  $(b,r)$ is a join of a directed set  $\{(x_i,r_i)\}_{i\in D}$ if and only if $b$  is a Yoneda limit of   $\{x_i\}_{i\in D}$ and $r$ is a limit (in the usual sense) of the net $\{r_i\}_{i\in D}$ of real numbers.
	
	We hasten to remark that \emph{standard quasi-metric spaces} are introduced in Goubault-Larrecq and Ng \cite[Definition 2.1]{Goubault2017} in a different way. Precisely, in \cite{Goubault2017} a quasi-metric space $(X, d)$ is said to be  standard if for each directed set of formal balls $\{(x_i, r_i)\}_{i\in D}$ and for each $s <\infty$, $\{(x_i, r_i)\}_{i\in D}$ has a join in $\mathrm{B}X $ if and only if $\{(x_i, s+r_i)\}_{i\in D}$ has a join in $\mathrm{B}X $. In the following we show that  standard quasi-metric spaces in the sense of  Definition \ref{defn of standard} coincide with those in the sense of \cite[Definition 2.1]{Goubault2017}.   Definition \ref{defn of standard}   emphasizes  the fact that for a standard quasi-metric space,   joins of directed sets of formal balls are computed componentwise.
	
	\begin{lemma}\label{yoneda complete = dcpo} Suppose $(X,d)$ is a quasi-metric space,   $\{(x_i,r_i)\}_{i\in D}$ is a directed set of $\mathrm{B}X $.  If 
		\begin{enumerate} 
			\item[\rm(i)] $\inf_{i\in D}r_i=0$, 
			\item[\rm(ii)] $(b,0)$ is a join of   $\{(x_i,r_i)\}_{i\in D}$, and \item[\rm(iii)] for each $t<\infty$ the directed set $\{(x_i,t+ r_i)\}_{i\in D}$ of $\mathrm{B}X $ has a join, 
		\end{enumerate} then  $b$ is a Yoneda limit  of  the forward Cauchy net $\{x_i\}_{i\in D}$.
	\end{lemma}
	
	\begin{proof}   
		We wish to show that for all $y\in X$, \[d(b,y)=\inf_{i\in D}\sup_{j\geq i}d(x_j,y).\]
		
		Fix $i\in D$. For each $j\geq i$, since $(x_j,r_j)\sqsubseteq(b,0)$, then \begin{align*}r_i+ d(b,y)&\geq r_j+ d(b,y) 
			\geq d(x_j,b)+ d(b,y)
			\geq d(x_j,y),\end{align*} hence $$r_i+ d(b,y)\geq \sup_{j\geq i}d(x_j,y).$$  
		Therefore,  \begin{align*}d(b,y)&=\inf_{i\in D}(r_i+ d(b,y))  \geq \inf_{i\in D}\sup_{j\geq i}d(x_j,y).\end{align*}
		
		For the converse inequality, let $$t=\inf_{i\in D}\sup_{j\geq i}d(x_j,y).$$ We wish to show that $t\geq d(b,y)$. We assume   $t<\infty$ and finish the proof in two steps. 
		
		{\bf Step 1}. $(b,t)$ is a join of $\{(x_i, t+ r_i)\}_{i\in D}$. 
		
		By assumption $\{(x_i, t+ r_i)\}_{i\in D}$  has a join, say $(z,s)$. We   show that   $(b,t)=(z,s)$.    Since $(b,t)$ is an upper bound of   $\{(x_i, t+ r_i)\}_{i\in D}$, then $(z,s)\sqsubseteq (b,t)$, hence $s\geq t$. Since $r_i$ tends to $0$ and $t+ r_i\geq s+ d(x_i,z)$ for all $i\in D$, it follows that $t\geq s$.  Therefore $s=t$ and $d(z,b)=0$. Since $(z,t)$ is a join of $\{(x_i, t+ r_i)\}_{i\in D}$,  then  $t+ r_i\geq t+ d(x_i,z)$  for all $i\in D$, hence  $r_i\geq d(x_i,z)$ for all $i\in D$. This shows that $(z,0)$ is an upper bound of $\{(x_i,r_i)\}_{i\in D}$, so   $(b,0)\sqsubseteq (z,0)$   and $d(b,z)=0$. Therefore, $b=z$.
		
		{\bf Step 2}.   $t\geq d(b,y)$. 
		
		For each $i\in D$,    \begin{align*}r_i+ t&= r_i+\inf_{k\in D}\sup_{j\geq k}d(x_j,y) \\
			&= r_i+\inf_{k\geq i}\sup_{j\geq k}d(x_j,y)  \\
			&= \inf_{k\geq i}\sup_{j\geq k}(r_i+ d(x_j,y)) \\
			&\geq \inf_{k\geq i}\sup_{j\geq k}(r_j+ d(x_i,x_j)+ d(x_j,y))\\ &\geq d(x_i,y),\end{align*} then $(x_i, t+ r_i)\sqsubseteq (y,0)$. This shows that $(y,0)$ is an upper bound of  $\{(x_i, t+ r_i)\}_{i\in D}$, then $(b,t)\sqsubseteq(y,0)$  and   $t\geq d(b,y)$, as desired. \end{proof}

	\begin{proposition} \label{equivalence of the defn} For each quasi-metric space $(X,d)$, the following are equivalent: \begin{enumerate} 
			\item[\rm(1)] $(X,d)$ is standard in the sense of  Definition \ref{defn of standard}. \item[\rm(2)] For each directed set $\{(x_i,r_i)\}_{i\in D}$ of $\mathrm{B}X$,    if   $\{(x_i,r_i)\}_{i\in D}$  has a join, then so do the following directed sets: \begin{itemize} 
				\item[\rm(i)] $\{(x_i,t+ r_i)\}_{i\in D}$ for each $t<\infty$; and \item[\rm(ii)] $\{(x_i,  r_i-t)\}_{i\in D}$ for $t=\inf_{i\in D}r_i$.\end{itemize} \end{enumerate}
	\end{proposition} 
	
	\begin{proof} 
		
		$(1)\Rightarrow(2)$ Lemma \ref{join of directed set in BX}.
		
		$(2)\Rightarrow(1)$  Suppose $\{(x_i,r_i)\}_{i\in D}$ is a directed set of $\mathrm{B}X$. We wish to show that if $(x,r)$ is a join of $\{(x_i,r_i)\}_{i\in D}$, then $x$ is a Yoneda limit of the forward Cauchy net $\{x_i\}_{i\in D}$. 
		Without loss of generality we may assume that $r_i>0$ for all $i\in D$. 
		
		Let $s=\inf_{i\in D}r_i$. Then $\{(x_i,  r_i-s)\}_{i\in D}$ is a directed set of $\mathrm{B}X$ with $\inf_{i\in D}(r_i-s)=0$. By (ii)   the directed set $\{(x_i, r_i-s)\}_{i\in D}$ has a join, say $(b,0)$. Applying (i) to the  directed set $\{(x_i, r_i-s)\}_{i\in D}$ yields that for each $t<\infty$, the directed set $\{(x_i,t+ r_i-s)\}_{i\in D}$ has a join. Thus the directed set $\{(x_i,r_i-s)\}_{i\in D}$ satisfies the conditions in Lemma \ref {yoneda complete = dcpo}, so $b$ is a Yoneda limit of the net $\{x_i\}_{i\in D}$. By Lemma \ref{join of directed set in BX}, $(b,s)$ is a join of $\{(x_i,r_i)\}_{i\in D}$. Since both $(x,r)$ and $(b,s)$ are join of the directed set $\{(x_i,r_i)\}_{i\in D}$, it follows that   $x=b$,  showing that $x$ is a Yoneda limit of   $\{x_i\}_{i\in D}$. \end{proof}

	It is clear that a quasi-metric space $(X,d)$ is standard in the sense of Goubault-Larrecq and Ng \cite[Definition 2.1]{Goubault2017} if and only if it satisfies (2) in Proposition \ref{equivalence of the defn}, so  the postulation of standard quasi-metric spaces in Definition \ref{defn of standard} agrees with that in \cite[Definition 2.1]{Goubault2017}. 
	
	\begin{theorem} Let $(X,d)$ be a quasi-metric space. \begin{enumerate}
			\item[\rm(i)] {\rm(Goubault-Larrecq and Ng \cite[Proposition 2.2]{Goubault2017})} If $(X,d)$ is Yoneda complete, then it is standard.
			\item[\rm(ii)]  {\rm(Kostanek and Waszkiewicz \cite[Theorem 7.1]{Kostanek2010})} $(X,d)$ is Yoneda complete if and only if $\mathrm{B}X$ is a dcpo.
	\end{enumerate} \end{theorem}
	
	\section{The presheaf monad}
	
	Following Lawvere \cite{Lawvere1973}, we view quasi-metric spaces as separated  categories enriched over the quantale $([0,\infty],\geq,+,0)$, non-expansive maps as functors between such categories. 
	This point of view has proved to be fruitful, see e.g. \cite{Bon-Rutten1997,Rutten1998,Vickers2005}. In particular, the quasi-metric $d_L$ on $[0,\infty]$ in Example \ref{internal hom} is  the ``internal hom'' on the base category, i.e., the quantale $([0,\infty],\geq,+,0)$.  For sake of self-containment, we recall in this section some basic ideas about the presheaf monad on the category of quasi-metric spaces and non-expansive maps.  We refer to Borceux \cite{Borceux1994a,Borceux1994b}, Kelly \cite{Kelly1982} and  Lawvere \cite{Lawvere1973} for category theory and enriched category theory.  
	
	Suppose $(X,d_X), (Y,d_Y)$ are quasi-metric spaces. A distributor $\phi\colon (X,d_X)\oto (Y,d_Y)$ is a map $\phi\colon X\times Y\to[0,\infty]$ such that for all $x,x'\in X$ and $y,y'\in Y$, $$d_Y(y,y')+\phi(x,y)+d_X(x',x)\geq \phi(x',y').$$ 
	
	For distributors $\phi\colon (X,d_X)\oto(Y,d_Y)$ and $\psi\colon(Y,d_Y)\oto (Z,d_Z)$, the composite  $\psi\circ\phi$ refers to the distributor $(X,d_X)\oto(Z,d_Z)$ given by $$\psi\circ\phi(x,z)=\inf_{y\in Y}(\psi(y,z)+\phi(x,y)).$$
	
	Let $f\colon (X,d_X)\to(Y,d_Y)$ be a non-expansive map. The graph of $f$ refers to the distributor  $$f_*\colon(X,d_X)\oto(Y,d_Y), \quad (x,y)\mapsto d_Y(f(x),y);$$   the cograph of $f$ refers to the distributor  $$f^*\colon(Y,d_Y)\oto(X,d_X), \quad (y,x)\mapsto d_Y(y,f(x)).$$  It is well-known that the graph $f_*$ is left adjoint to the cograph $f^*$ in the sense that $$f^*\circ f_*(x,x')\leq d_X(x,x')\quad \text{and}\quad f_*\circ f^*(y,y')\geq d_Y(y,y')$$ for all $x,x'\in X$ and $y,y'\in Y$.
	
	A \emph{weight} of a quasi-metric space $(X,d)$ is a distributor $\phi\colon(X,d)\oto\star$, where $\star$ denotes the singleton metric space. In other words, a weight of $(X,d)$ is a map $\phi\colon X\longrightarrow[0,\infty]$ such that  $\phi(y)+d(x,y)\geq\phi(x)$ for all $x,y\in X$.  It is clear that a weight of $(X,d)$ is exactly a non-expansive map $(X,d^{\rm op})\longrightarrow ([0,\infty],d_L)$; or equivalently, a non-expansive map $(X,d)\longrightarrow ([0,\infty],d_R)$.

	For each quasi-metric space $(X,d)$, the following holds: \begin{itemize}
		\item   For each $x\in X$, $d(-,x)$ is a weight of $(X,d)$.  Weights  of this form are said to be \emph{representable}.
		\item  If each   $\phi_i$ $(i\in J)$  is a weight of $(X,d)$, then so are $\inf_{i\in J}\phi_i$ and $\sup_{i\in J}\phi_i$.
		\item  If $\phi$ is a  weight of $(X,d)$, then so are $r+\phi$ and $\phi\ominus r$ for all $r\in [0,\infty]$.
	\end{itemize}
	
	We write $$\mathcal{P}X$$ for the set of all weights of $(X,d)$ and equip it with the quasi-metric $\rho_X$ inherited from $([0,\infty]^X,\rho_X)$; that is, for all weights $\phi$ and $\psi$,  $$   \rho_X(\phi,\psi)=\sup_{x\in X}\psi(x)\ominus \phi(x).$$  The quasi-metric space $(\mathcal{P} X,\rho_X)$ is an analogue of the category of presheaves in the $[0,\infty]$-enriched context.  
	
	The following lemma ensures the map
	$${\sf y}_X\colon (X,d)\longrightarrow (\mathcal{P}X,\rho_X), \quad {\sf y}_X(x)=d(-,x)$$
	is isometric, it is the \emph{Yoneda embedding} in the $[0,\infty]$-enriched context.
	\begin{lemma}[Yoneda lemma] \label{yoneda lemma}
		Let $(X,d)$ be a quasi-metric space. Then for each  $x\in X$ and each    $\phi\in\mathcal{P} X$,
		$$\rho_X(d(-,x),\phi)=\phi(x).$$
	\end{lemma}

	Lwt $f\colon (X,d_X)\to (Y,d_Y)$ and $g\colon (Y,d_Y)\to (X,d_X)$ be non-expansive maps. We say that $f$ is left adjoint to $g$, or $g$ is right adjoint to $f$, and write $$f\dashv g\colon (Y,d_Y)\to (X,d_X),$$  if for all $x\in X$ and $y\in Y$, $$d_Y(f(x),y)=d_X(x,g(y)).$$ It is not hard to check that $f$ is left adjoint to $g$ if and only if $f\colon (X,\sqsubseteq_X)\to (Y,\sqsubseteq_Y)$ is left adjoint to $g\colon (Y,\sqsubseteq_Y)\to (X,\sqsubseteq_X)$, see e.g. \cite{Lai-Zhang2007}.
	
	Each non-expansive map $f\colon (X,d_X)\to (Y,d_Y)$ gives rise to an adjunction $$f^\rightarrow\dashv f^\leftarrow\colon (\mathcal{P}Y,\rho_Y)\longrightarrow(\mathcal{P}X,\rho_X),$$  where
	$$f^\rightarrow\colon (\mathcal{P}X,\rho_X)\longrightarrow(\mathcal{P}Y,\rho_Y), \quad f^{\rightarrow}(\phi)=\phi\circ f^* $$ and 
	$$f^\leftarrow\colon (\mathcal{P}Y,\rho_Y)\longrightarrow(\mathcal{P}X,\rho_X), \quad f^\leftarrow(\psi)=\psi\circ f_*.$$

	The assignment $f\mapsto f^\rightarrow$ defines a functor $$\mathcal{P}\colon{\sf qMet}\longrightarrow{\sf qMet},$$   called the presheaf functor on the category $\sf qMet$, which plays an important role in the study of quasi-metric spaces.

	Let  $f\colon (K,d_K)\longrightarrow (X,d_X)$ be a non-expansive map and  let $\phi$ be a weight of $(K,d_K)$. The \emph{colimit of $f$ weighted by $\phi$}, written $\colim_\phi f$, is an element $b$ of $X$ such that for all $x\in X$, $$d_X(b,x)= \rho_X(\phi\circ f^*,d_X(-,x))=\rho_K(\phi,d_X(f(-),x)). $$  The colimit of the identity map $X\longrightarrow X$ weighted by $\phi$, when exists, will also be called the \emph{colimit of $\phi$} by abuse of language, and denoted by $\colim\phi$ instead of $\colim_\phi 1_X$.   
	
	A quasi-metric space $(X,d)$ is \emph{cocomplete} if all weighted colimits of non-expansive maps to $(X,d)$ exist.  For a quasi-metric space $(X,d)$ the following are equivalent: \begin{enumerate} 
		\item[\rm(1)] $(X,d)$ is cocomplete. \item[\rm(2)] Every weight of $(X,d)$ has a colimit. \item[\rm(3)] The Yoneda embedding ${\sf y}_X\colon (X,d)\longrightarrow(\mathcal{P}X,\rho_X)$ has a left adjoint. \end{enumerate}
	
	The dual notion of colimit is that of limit, which we spell out here for convenience of the reader. A \emph{coweight} of a quasi-metric space $(X,d)$ is a distributor $\psi\colon\star\oto(X,d)$. In other words, a coweight of $(X,d)$ is a non-expansive map $\psi\colon(X,d)\longrightarrow ([0,\infty],d_L)$. 
	
	Let  $f\colon (K,d_K)\longrightarrow (X,d_X)$ be a non-expansive map and   $\psi$ a coweight of $(K,d_K)$. The \emph{limit of $f$ weighted by $\psi$}, written $\lim_\psi f$, is an element $a$ of $X$ such that for all $x\in X$, $$d_X(x,a) =\rho_K(\psi,d_X(x,f(-))). $$
	
	A quasi-metric space $(X,d)$ is \emph{complete} if all weighted limits of non-expansive maps to $(X,d)$ exist. It is  known that a quasi-metric space is cocomplete if and only if it is complete, see e.g. \cite{Stubbe2005}.
	
	\begin{example}For each quasi-metric space $(X,d)$, the space $(\mathcal{P} X,\rho_X)$ is both cocomplete and complete. This is a special case of a general result in enriched category theory, see e.g.  \cite{Stubbe2005,Stubbe2006}. 
		
		For each weight $\Phi$ of $(\mathcal{P} X,\rho_X)$, the  colimit of $\Phi$ is given by $$\colim\Phi=\inf_{\phi\in\mathcal{P} X}(\Phi(\phi)+\phi).$$  
		
		For each quasi-metric space $(X,d)$, the space $(\mathcal{P} X,\rho_X)$ is complete. For each coweight $\Psi$ of $(\mathcal{P} X,\rho_X)$, the  limit of (the identity map weighted by) $\Psi$ is given by  $$\lim\Psi=\sup_{\phi\in\mathcal{P} X}(\phi\ominus\Psi(\phi)).$$ 
	\end{example}
	
	Assigning to each quasi-metric space $(X,d)$ the Yoneda embedding ${\sf y}_X\colon (X,d)\longrightarrow (\mathcal{P}X,\rho_X)$ defines a natural transformation $${\sf y}\colon {\sf Id}\longrightarrow\mathcal{P}$$ from the identity functor to the presheaf functor $\mathcal{P}$.
	Assigning to each quasi-metric space $(X,d)$ the left adjoint $${\sf m}_X\colon(\mathcal{P}^2 X,\rho_{\mathcal{P} X})\longrightarrow(\mathcal{P} X,\rho_X), \quad \Phi\mapsto \inf_{\phi\in\mathcal{P} X}(\Phi(\phi)+\phi)$$ of the Yoneda embedding $${\sf y}_{\mathcal{P} X}\colon (\mathcal{P} X,\rho_X)\longrightarrow(\mathcal{P}^2 X,\rho_{\mathcal{P} X})$$ gives rise to a natural transformation $$\sf m\colon\mathcal{P}^2\longrightarrow\mathcal{P}.$$
	The triple $$\mathbb{P}=(\mathcal{P},{\sf m},{\sf y})$$ is a monad on the category $\sf qMet$, called the presheaf monad. The category of algebras of the monad  $(\mathcal{P},{\sf m},{\sf y})$ is the category of cocomplete (hence complete) quasi-metric spaces and non-expansive maps preserving colimits. The monad $\mathbb{P}=(\mathcal{P},{\sf m},{\sf y})$  is a Kock-Z\"{o}berlein type monad, or a lax-idempotent monad in the terminology of \cite{KL1997}. Kock-Z\"{o}berlein type monads are a kind of monads on 2-categories.  Besides the papers of Kock   and Z\"{o}berlein \cite{Kock,Zo76},  ideas and properties about such monads can be found in  \cite{Hof2013,KL1997}.

	A subfunctor of the presheaf functor $\mathcal{P}$ is a functor $\mathcal{T}\colon{\sf qMet}\longrightarrow{\sf qMet}$ such that \begin{itemize}\item for each quasi-metric space  $(X,d)$, $\mathcal{T}X$ is a subset of $\mathcal{P} X$ equipped with  $\rho_X$;   \item for each non-expansive map $f\colon (X,d_X)\to (Y,d_Y)$ and each $\phi\in\mathcal{T}X$, the weight $f^\rightarrow(\phi)$ belongs to $\mathcal{T}Y$ and $\mathcal{T}f(\phi)=f^\rightarrow(\phi)$. \end{itemize}
	
	\begin{definition} {\rm(\cite{Kostanek2010,Lai-Zhang2007,Waszkiewicz2009})} Consider the category  {\sf qMet} of quasi-metric spaces and non-expansive maps. \begin{enumerate}
			\item[\rm(i)] A class of weights  on {\sf qMet} is a subfunctor $\mathcal{T}\colon{\sf qMet}\longrightarrow{\sf qMet}$ of the presheaf functor $\mathcal{P}$ such that for each quasi-metric space $(X,d)$ and each $x\in X$, the representable weight $d(-,x)$ belongs to $\mathcal{T} X$. In this case, the Yoneda embedding factors through the inclusion natural transformation $\mathfrak{i}\colon \mathcal{T}\longrightarrow\mathcal{P} $.  
			
			\item[\rm(ii)]  A    class of weights  $\mathcal{T}\colon{\sf qMet}\longrightarrow{\sf qMet}$ is saturated provided that for  each quasi-metric space $(X,d)$, $\mathcal{T} X$ is closed in $\mathcal{P} X$ under formation of $\mathcal{T}$-colimits in the sense that for each $\Phi\in\mathcal{T}\mathcal{T} X$, the colimit of the inclusion map $\mathfrak{i}_X\colon \mathcal{T} X\longrightarrow \mathcal{P} X$ weighted by $\Phi$ belongs to $\mathcal{T} X$. In this case, $\colim_\Phi \mathfrak{i}_X$ is a colimit of the weight $\Phi$ of $(\mathcal{T} X,\rho_X)$.
	\end{enumerate}     \end{definition}
	
	Let $\mathcal{T}$ be a class of weights on $\sf qMet$. \begin{itemize} \item A quasi-metric space $(X,d)$ is  \emph{$\mathcal{T}$-cocomplete} if, for each non-expansive map $g\colon(Z,d_Z)\longrightarrow(X,d_X)$ and each $\xi\in\mathcal{T}Z$, the colimit of $g$ weighted by $\xi$ always exists. This is equivalent to that  each $\phi\in \mathcal{T}X$ has a colimit.  
		\item A non-expansive map $f\colon(X,d_X)\longrightarrow(Y,d_Y)$ \emph{preserves $\mathcal{T}$-colimits} if, for each non-expansive map $g\colon(Z,d_Z)\longrightarrow(X,d_X)$ and each $\xi\in\mathcal{T}Z$, it holds that $$f({\colim}_\xi g)= {\colim}_\xi (f\circ g)$$ whenever $\colim_\xi g$ exists. This is equivalent to that for each $\phi\in \mathcal{T}X$, $f(\colim \phi)$ is a colimit of $f$ weighted by $\phi$ whenever $\colim\phi$ exists.  
	\end{itemize}

	Let $\mathcal{T}$ be a class of weights on $\sf qMet$; let $\mathfrak{i} \colon \mathcal{T}  \longrightarrow \mathcal{P}  $ be the inclusion natural transformation; and let $\mathfrak{i} *\mathfrak{i}$ be the horizontal composite of $\mathfrak{i}$ with itself. Then, $\mathcal{T}$ is saturated if and only if the natural transformation ${\sf m}\circ(\mathfrak{i} *\mathfrak{i} )$ factors (uniquely) through $\mathfrak{i}$,  because 
	$$(\mathfrak{i} *\mathfrak{i} )_X(\Phi)=\Phi\circ \mathfrak{i}_X^* \quad \text{and}\quad  {\colim}_\Phi\mathfrak{i}_X = {\colim} (\Phi\circ \mathfrak{i}_X^*)$$ for all $\Phi\in\mathcal{T}\mathcal{T} X$. 
	\[\bfig \square [\mathcal{T}^2`\mathcal{T} `\mathcal{P}^2 `\mathcal{P} ; {\sf m}`\mathfrak{i} *\mathfrak{i} `\mathfrak{i}` {\sf m}] \efig\]   
	Therefore, a saturated class of weights is a submonad $$\mathbb{T}=(\mathcal{T},{\sf m},{\sf y} )$$ of the presheaf monad $\mathbb{P}=(\mathcal{P},{\sf m},{\sf y})$. We also use {\sf y} and  ${\sf m}$, respectively, to denote the  natural transformations ${\sf Id}\longrightarrow\mathcal{T}$ and $\mathcal{T}^2\longrightarrow\mathcal{T}$through which the unit and the multiplication of the presheaf monad factor, respectively. This won't lead to confusion.  
	
	For a saturated class of weights $\mathcal{T}$, the monad $\mathbb{T}=(\mathcal{T},{\sf m},{\sf y})$  is of Kock-Z\"{o}berlein type, since so is $\mathbb{P}=(\mathcal{P},{\sf m},{\sf y})$. Monads  of Kock-Z\"{o}berlein type have many pleasant properties (see e.g.  \cite{Hof2013,KL1997}), some of them are listed here in the form for $\mathbb{T}=(\mathcal{T},{\sf m},{\sf y})$:   \begin{enumerate} 
		\item[\rm(i)]  A quasi-metric space $(X,d)$ is a  $\mathbb{T}$-algebra, if and only if it is $\mathcal{T}$-cocomplete, if and only if the Yoneda embedding ${\sf y}\colon (X,d)\longrightarrow(\mathcal{T} X,\rho_X)$ (with codomain restricted to $(\mathcal{T} X,\rho_X)$) has a left adjoint. \item[\rm(ii)]  A non-expansive map between $\mathbb{T}$-algebras is a $\mathbb{T}$-homomorphism if and only if it preserve $\mathcal{T}$-colimits.  \item[\rm(iii)] For each quasi-metric space $(X,d)$, we have a string of adjunctions $$\mathcal{T}\mathsf{y}_X  \dashv\mathsf{m}_X\dashv\mathsf{y}_{\mathcal{T}X}\colon  (\mathcal{T}X,\rho_X)\longrightarrow (\mathcal{TT}X,\rho_{\mathcal{T}X}).$$\end{enumerate} 
	
	Typical submonads of $(\mathcal{P},{\sf m},{\sf y})$ include the formal ball monad \cite{Goubault2019}, the ideal monad \cite{Lai-Zhang2020}, the bounded ideal monad (to be introduced in the next section), and others. 
	
	For each quasi-metric space $(X,d)$, let $\mathcal{G}X$ be the subspace of $(\mathcal{P} X,\rho_X)$ consisting of weights of the form $r+d(-,x)$, $r>0$, $x\in X$. The assignemnt $(X,d)\mapsto (\mathcal{G}X,\rho_X)$ defines a saturated class of weights on $\sf qMet$, hence a submonad of $(\mathcal{P},{\sf m},{\sf y})$. The resulting monad is precisely the \emph{formal ball monad} of Goubault-Larrecq \cite{Goubault2019}. Since $$ \rho_X(r+d(-,x),s+d(-,y))= (s+d(x,y))\ominus r,$$  the ordered set $(\mathrm{B}X,\sqsubseteq)$ of formal balls is essentially the underlying ordered set of the quasi-metric space $(\mathcal{G}X,\rho_X)$.
	
	\begin{definition}{\rm(\cite{Lai-Zhang2020})} An ideal  of a quasi-metric space $(X,d)$ is a weight $\phi$ such that \begin{enumerate}
			\item[\rm (i)] $\inf_{x\in X} \phi(x)=0$;   
			\item[\rm (ii)]  $\rho_X(\phi,\lambda\wedge \mu)=\rho_X(\phi,\lambda)\wedge \rho_X(\phi, \mu)$ for all $\lambda,\mu \in \mathcal{P}X$.
	\end{enumerate}\end{definition}
	
	For each quasi-metric space $(X,d)$, let $\mathcal{I} X$ be the subspace of $(\mathcal{P} X,\rho_X)$ consisting of ideals of $(X,d)$. Then the assignment $(X,d)\mapsto (\mathcal{I} X,\rho_X)$ defines a saturated class of weights on $\sf qMet$, the corresponding monad $$\mathbb{I}=(\mathcal{I},{\sf m},{\sf y})$$ is called the \emph{ideal monad} on $\sf qMet$. The following two lemmas show that the category of $\mathbb{I}$-algebras is precisely the category of Yoneda complete quasi-metric spaces and Yoneda continuous maps.
	
	\begin{lemma} {\rm(\cite{Lai-Zhang2020})} \label{for-weight}
		For each weight $\phi$ of a quasi-metric space $(X,d)$, the following are equivalent:
		\begin{enumerate}
			\item[\rm (1)] $\phi$ is an ideal.
			\item[\rm (2)] $\inf_{x\in X} \phi(x)=0$ and $\mathrm{B}\phi\coloneqq \{(x,r) \mid \phi(x)<r\}$ is a directed set of $\mathrm{B}X$.
			\item[\rm (3)] There is a forward Cauchy net $\{x_i\}_{i\in D}$ of $(X,d)$ such that 
			$$\phi =\inf_{i\in D} \sup_{j\geq i} d(-,x_j).$$ 
		\end{enumerate}
	\end{lemma}

	\begin{lemma} {\rm(Flagg,   S\"{u}nderhauf and  Wagner \cite{Flagg1996})} \label{colimit=yoneda limit}
		Suppose that $\{x_i\}_{i\in D} $ is a forward Cauchy net of a quasi-metric space $(X,d_X)$;  $\phi=\inf_{i\in D} \sup_{j\geq i} d_X(-,x_j)$ is the ideal generated by $\{x_i\}_{i\in D} $; and that $f\colon(X,d_X)\longrightarrow(Y,d_Y)$ is a non-expansive map. Then \begin{enumerate}
			\item[\rm (i)] $x\in X$ is a Yoneda limit of $\{x_i\}_{i\in D}$ if and only if it is a colimit of $\phi$.  \item[\rm (ii)] $f^\rightarrow(\phi)$ is the ideal generated by the forward Cauchy net $\{f(x_i)\}_{i\in D}$ of $(Y,d_Y)$.  \end{enumerate}
	\end{lemma} 
	
	\begin{remark} Lemma \ref{for-weight} shows that ideals of quasi-metric spaces are precisely the weights generated by forward Cauchy nets.  Such weights have been introduced in the quantale-enriched setting in different forms,  see e.g. \cite{Flagg1996,Vickers2005,Wagner1997}. In particular, for quasi-metric spaces, these ideals coincide with the \emph{flat left modules} of Vickers \cite{Vickers2005}. For a comparison of various generalizations of ideals of partially ordered sets in the quantale-enriched setting, see \cite{Lai-Zhang2020}. \end{remark}

	\section{The bounded ideal monad} 
	
	\begin{definition}
		Suppose $(X,d)$ is a quasi-metric space.  A weight  $\phi$  of $(X,d)$ is  bounded if $$\inf_{x\in X}\rho_X(\phi, d(-,x))<\infty.$$ 
	\end{definition}
	
	In other words,   $\phi$ is  bounded if   $ r+\phi\geq d(-,a) $ for some $r<\infty$ and some $a\in X$. 
	Every weight $\phi$  with a colimit is bounded, for if $a$ is a colimit of $\phi$, then $0=d(a,a)= \rho_X(\phi,d(-,a))$, hence $\phi\geq d(-,a)$. 
	
	\begin{proposition}\label{B is saturated} Bounded weights form a saturated class of weights on the category of quasi-metric spaces and non-expansive maps. \end{proposition}
	
	\begin{proof} First, we show that bounded weights form a class of weights. That means, if $f\colon(X,d_X)\longrightarrow(Y,d_Y)$ is a non-expansive map and $\phi$ is a bounded weight of $(X,d_X)$, then $f^\rightarrow(\phi)$ is a bounded weight of $(Y,d_Y)$. This follows from the fact that if $r+\phi\geq d_X(-,a) $, then $r+f^\rightarrow(\phi)\geq d_Y(-,f(a)) $.
		
		Second, we show that the class of bounded weights is saturated.  For each quasi-metric space $(X,d)$, let $(\mathcal{B}X,\rho_X)$ be the subspace of $(\mathcal{P} X,\rho_X)$, where $\mathcal{B}X$ is the set of all bounded weights of $(X,d)$.  We need to show that for each bounded weight $\Phi$ of $(\mathcal{B}X,\rho_X)$, the colimit $\colim_\Phi \mathfrak{i}_X$ of the inclusion map $\mathfrak{i}_X\colon\mathcal{B}X\longrightarrow\mathcal{P} X$ weighted by $\Phi$ is a bounded weight of $(X,d)$. 
		
		A direct calculation gives that $${\colim}_\Phi \mathfrak{i}_X=\inf_{\phi\in\mathcal{B}X}(\Phi(\phi)+\phi).$$ Since $\Phi$ is bounded, there exist $\psi\in \mathcal{B}X$ and $r<\infty$ such that 
		$r+\Phi\geq \rho_X(-,\psi)$. Since $\psi$ is bounded,   there exist  $a\in X$ and $s<\infty$ such that $s+\psi\geq d(-,a)$. Then 
		\begin{align*} s+ r +{\colim}_\Phi \mathfrak{i}_X&=s+r+\inf_{\phi\in\mathcal{B}X}(\Phi(\phi)+\phi)\\ &= s+ \inf_{\phi\in\mathcal{B}X}(r+\Phi(\phi)+\phi) \\ 
			&\geq s+\inf_{\phi\in\mathcal{B}X}(\rho_X(\phi,\psi)+\phi)\\ 
			&=s+\psi\\
			&\geq d(-,a), \end{align*} hence $\colim_\Phi \mathfrak{i}_X$ is a bounded weight of $(X,d)$.
	\end{proof}
	
	\begin{definition}
		A weight $\phi$ of a quasi-metric space $(X,d)$ is called a bounded ideal if it is at the same time an ideal and a bounded weight of $(X,d)$.    
	\end{definition} 
	
	We write $(\mathcal{J}X,\rho_X)$ for the subspace of $(\mathcal{P} X,\rho_X)$ consisting of bounded ideals. It is clear that if $f\colon(X,d_X)\longrightarrow(Y,d_Y)$ is a non-expansive map and $\phi$ is a bounded ideal of $(X,d_X)$, then $f^\rightarrow(\phi)$ is a bounded ideal of $(Y,d_Y)$. So, assigning to each quasi-metric space $(X,d)$ the space $(\mathcal{J}X,\rho_X)$ of bounded ideals defines  a class of weights $\mathcal{J}\colon\sf qMet \longrightarrow qMet$.
	
	\begin{proposition} \label{saturated}
		Bounded ideals form a saturated class of weights on the category of quasi-metric spaces and non-expansive maps. 
	\end{proposition} 
	
	\begin{proof} It suffices to show that for each bounded ideal $\Phi$ of $(\mathcal{J}X,\rho_X)$, the colimit $\colim_\Phi \mathfrak{i}_X$ of the inclusion map $\mathfrak{i}_X\colon\mathcal{J}X\longrightarrow\mathcal{P} X$ weighted by $\Phi$ is a bounded ideal of $(X,d)$. Since $${\colim}_\Phi \mathfrak{i}_X=\inf_{\phi\in\mathcal{J}X}(\Phi(\phi)+\phi),$$ the proof proceeds in the same way as that in Proposition \ref{B is saturated} and \cite[Proposition 4.4]{Lai-Zhang2020}, so we omit the details here. \end{proof}
	
	Write $$\mathbb{J}=(\mathcal{J},{\sf m},{\sf y})$$ for the submonad of the presheaf monad that corresponds to the class of bounded ideals. In order to characterize $\mathbb{J}$-algebras, we need the notion of local dcpos introduced in Mislove \cite{Mislove1999}.
	A partially ordered set $P$ is a \emph{local dcpo}  if each directed set of $P$ with an upper bound has a join. 
	
	\begin{theorem} \label{J-algebra}
		A quasi-metric space $(X,d)$ is a $\mathbb{J}$-algebra if and only if $(X,d)$ is standard and $(\mathrm{B}X,\sqsubseteq)$ is a local dcpo.
	\end{theorem} 
	
	\begin{proof} We prove the necessity first. 
		Let $(X,d)$ be a $\mathbb{J}$-algebra.  
		
		Suppose $\{(x_i,r_i)\}_{i\in D}$ is a directed set of $\mathrm{B}X$ with an upper bound, say $(a,t)$. Consider the forward Cauchy net $\{x_i\}_{i\in D}$ and the ideal $$\phi=\inf_{i\in D}\sup_{j\geq i} d(-,x_j) $$ of $(X,d)$.
		We claim that  $\phi $ is  bounded. Since $(a,t) $ is an upper bound of $\{(x_i,r_i)\}_{i\in D}$, we have $r_i\geq t+d(x_i,a)\geq d(x_i,a)$ for all $i\in D$.  We may assume that $\{r_i\}_{i\in D}$ has an upper bound, say $G$; otherwise take some $k\in D$ and consider the directed set $\{(x_i,r_i)\}_{i\geq k}$. Then \begin{align*} G+ \phi &= \inf_{i\in D}\sup_{j\geq i} (G+d(-,x_j))\\ &\geq \inf_{i\in D}\sup_{j\geq i} (d(x_j,a)+d(-,x_j)) \\ &\geq d(-,a),\end{align*} so $\phi$ is bounded.
		Since $(X,d)$ is a $\mathbb{J}$-algebra, $\phi$ has a colimit, then the forward Cauchy net $\{x_i\}_{i\in D}$ has a Yoneda limit by Lemma \ref{colimit=yoneda limit}, hence the directed set $\{(x_i,r_i)\}_{i\in D}$ has a join by Lemma \ref{join of directed set in BX}.
		This proves that $(\mathrm{B}X,\sqsubseteq)$ is a local dcpo.	
		
		To see that $(X,d)$ is standard, assume that $\{(x_i,r_i)\}_{i\in D}$ is a directed set of $(\mathrm{B}X,\sqsubseteq)$ that has a join. The above argument shows that the ideal  $$\phi=\inf_{i\in D}\sup_{j\geq i} d(-,x_j) $$ of $(X,d)$ is bounded, thus it has a colimit, which is by Lemma \ref{colimit=yoneda limit} a Yoneda limit of $\{x_i\}_{i\in D}$. 
		
		Now we prove the sufficiency. Let $(X,d)$ be a standard quasi-metric space such that $(\mathrm{B}X,\sqsubseteq)$ is a local dcpo. We wish to show that $(X,d)$ is a $\mathbb{J}$-algebra; that is, each bounded ideal $\phi$ of $(X,d)$ has a colimit. 
		
		Since $\phi$ is bounded, there exist $a\in X$ and $t<\infty$ such that $t+\phi\geq d(-,a)$. Since $\phi$ is an ideal, $$\mathrm{B}\phi\coloneqq \{(x,r) \mid \phi(x)<r\}$$ is a directed set of $\mathrm{B}X$  by Lemma \ref{for-weight}. We index $\mathrm{B}\phi$ by itself as $\{(x_i,r_i)\}_{i\in D}$. It is clear that $\{(x_i,t+r_i)\}_{i\in D}$ is a directed set of $\mathrm{B}X$. Since for all $i\in D$, $$t+r_i>t+\phi(x_i)\geq d(x_i,a),$$ it follows that $(x_i,t+r_i)\sqsubseteq(a,0)$, hence   $(a,0)$ is an upper bound of $\{(x_i,t+r_i)\}_{i\in D}$.  Since $\mathrm{B}X$ is a local dcpo, $\{(x_i,t+r_i)\}_{i\in D}$ has a join, say $(b,s)$. Since $(X,d)$ is standard, then $b$ is a Yoneda limit of the forward Cauchy net $\{x_i\}_{i\in D}$. Finally, since $$\phi=\inf_{i\in D}\sup_{j\geq i} d(-,x_j), $$ it follows that $b$ is a colimit of $\phi$ by Lemma \ref{colimit=yoneda limit}.  \end{proof} 
	
	There is a characterization of $\mathbb{J}$-algebras by forward Cauchy nets. A forward Cauchy net  $\{x_i\}_{i\in D}$ of a quasi-metric space $(X,d)$ is \emph{bounded} provided that there exist some $a\in X$ and $r<\infty$ such that for each $x\in X$ and each $i\in D$ there is some $j\geq i$ for which $r+d(x,x_j)\geq d(x,a)$; or equivalently, the ideal  $ \phi\coloneqq \inf_{i\in D}\sup_{j\geq i}d(-,x_j)$  is bounded. From the argument of Theorem \ref{J-algebra} one easily deduces the following:
	
	\begin{proposition} A quasi-metric space is a $\mathbb{J}$-algebra if and only if each of its bounded forward Cauchy nets has a Yoneda limit. \end{proposition}
	
	In a recent paper \cite{Lu2023}, Lu and Zhao introduced the notion of local Yoneda complete quasi-metric spaces.  A quasi-metric space $(X,d)$ is  \emph{local Yoneda complete}  if it satisfies: 
	\begin{enumerate}
		\item[\rm(i)]  If $(x,r)$ is a join of a directed set $\{(x_i,r_i)\}_{i\in D}$ of $(\mathrm{B}X,\sqsubseteq)$, then $r=\inf_{i\in D} r_i$. 
		\item[\rm(ii)] $(\mathrm{B}X,\sqsubseteq)$ is a local dcpo.
	\end{enumerate}

	It is proved in \cite[Proposition 15]{Lu2023} that each local Yoneda complete quasi-metric space is standard, so,  $\mathbb{J}$-algebras  are precisely  local Yoneda complete quasi-metric spaces in the sense of  Lu and Zhao.

	Let $(X,d_X)$ and $(Y,d_Y)$ be $\mathbb{J}$-algebras. A homomorphism $f\colon(X,d_X)\longrightarrow(Y,d_Y)$ is by definition a non-expansive map that  preserves colimits of bounded ideals.  As usual, we write  $$\mathbb{J}\text{-}{\sf Alg}$$ for the category of $\mathbb{J}$-algebras and homomorphisms. 
	
	Since bounded ideals form a saturated class of weights,  for each quasi-metric space $(X,d)$ the space $(\mathcal{J}X,\rho_X)$ is a $\mathbb{J}$-algebra, and the assignment $(X,d)\mapsto (\mathcal{J}X,\rho_X)$  defines a functor $$\mathcal{J}\colon{\sf qMet}\longrightarrow \mathbb{J}\text{-}{\sf Alg}.$$  The following conclusion is an instance of a general result of Kelly \cite[Theorem 5.35]{Kelly1982}.
	
	\begin{theorem} \label{free cocompletion} {\rm(\cite{Kelly1982,KS2005})} The functor $ \mathcal{J}\colon{\sf qMet}\longrightarrow \mathbb{J}\text{-}{\sf Alg} $ is left adjoint to the forgetful functor $ U\colon \mathbb{J}\text{-}{\sf Alg}\longrightarrow {\sf qMet}$. The space $(\mathcal{J}X,\rho_X)$ is  the  free $\mathcal{J}$-cocompletion of $(X,d)$. \end{theorem}

	Since every ideal with a colimit is  bounded, it follows that a non-expansive map preserves colimits of bounded ideals, if and only if it preserves colimits of ideals, if and only if it is Yoneda continuous. We write $${\sf qMet}^\uparrow$$ for the category of quasi-metric spaces and Yoneda continuous non-expansive maps. This category is an analogue, in the $[0,\infty]$-enriched context, of the category ${\sf Poset}^\uparrow$ of partially ordered sets and Scott continuous functions. The category $\mathbb{J}\text{-}{\sf Alg}$ of $\mathbb{J}$-algebras and homomorphisms is  a full subcategory of ${\sf qMet}^\uparrow$,  indeed a reflective one. That means, the inclusion functor $$  \mathbb{J}\text{-}{\sf Alg}\longrightarrow {\sf qMet}^\uparrow$$ has a left adjoint. This has been proved in \cite[Section 4]{Lu2023}, it is an extension of a result for local dcpos in  \cite{Mislove1999,Zhao-Fan2010} to quasi-metric spaces. Actually, there holds a general conclusion, as we see below. 
	
	Let $\mathcal{T}$ be a saturated class of weights on $\sf qMet$  and let $${\sf qMet}^{\sf t}$$ be the category of quasi-metric spaces and   non-expansive maps preserving $\mathcal{T}$-colimits.  Theorem \ref{idempotent completion} below says that  the inclusion functor  $$ \mathbb{T}\text{-}{\sf Alg}\longrightarrow {\sf qMet}^{\sf t}$$ has a left adjoint. In other words, $\mathbb{T}\text{-}{\sf Alg}$ is a reflective full subcategory of ${\sf qMet}^{\sf t}$.
	
	Let $\mathcal{T}$ be a saturated class of weights on $\sf qMet$. A weight $\phi$ of a quasi-metric space $(X,d)$ is  \emph{$\mathcal{T}$-closed} provided that for each   $\psi\in\mathcal{T} X$,    $$\rho_X(\psi,\phi)\geq \phi(\colim \psi)$$ whenever $\psi$ has a colimit. 
	
	When we view a weight $\phi$ of a quasi-metric space $(X,d)$ as a non-expansive map $$\phi\colon(X,d)\longrightarrow([0,\infty],d_R),$$  it is not hard to check that for each $\psi\in\mathcal{T}X$, the colimit of $\phi$ weighted by $\psi$ is given by $\rho_X(\psi,\phi)$, which is to say, $${\colim}_\psi\phi=\inf_{x\in X}(\phi(x)\ominus\psi(x))= \rho_X(\psi,\phi).$$ Since the underlying order of $([0,\infty],d_R)$ is the usual order relation $\leq$ of real numbers, it follows that $\phi(\colim\psi)\geq \colim_\psi\phi$. Therefore, a weight $\phi$ of a quasi-metric space $(X,d)$ is  $\mathcal{T}$-closed if and only if, as a  non-expansive map  $(X,d)\longrightarrow([0,\infty],d_R)$, it preserves $\mathcal{T}$-colimits.
	
	Let $$\sigma X$$ be  the subspace of $(\mathcal{P} X,\rho_X)$ consisting of  $\mathcal{T}$-closed weights of $(X,d)$. Then
	\begin{enumerate} \item[\rm (i)] every representable weight   $d(-,x)$ is $\mathcal{T}$-closed; 
		\item[\rm(ii)]  for each subset $\{\phi\}_{i\in J}$ of $\sigma X$, $\sup_{i\in J} \phi_i \in \sigma X$; \item[\rm(iii)] for all $\phi \in \sigma X$ and   $r\in [0,\infty]$,   $\phi\ominus r\in\sigma X$.\end{enumerate} 
	Therefore, the subset $\sigma X$ is closed under formation of limits in $(\mathcal{P} X,\rho_X)$ in the sense that for each coweight $\Psi$ of $(\sigma X,\rho_X)$, the limit of the inclusion map $\sigma X\longrightarrow \mathcal{P} X$ weighted by $\Psi$ belongs to $\sigma X$. In particular, the space $(\sigma X,\rho_X)$ is complete, hence cocomplete. 
	
	\begin{remark} In the case that $\mathcal{T}$ is the class of ideals or the class of bounded ideals, for each quasi-metric space $(X,d)$  the set $\sigma X$ also satisfies that $\phi_1,\phi_2\in\sigma X\Rightarrow \phi_1\wedge\phi_2\in\sigma X$. In these cases, $\sigma X$ is the set of lower regular functions of the Scott approach structure of $(X,d)$. Scott approach structure is an analogue of Scott topology in the $[0,\infty]$-enriched context, for related notions the reader is referred to the monograph  \cite{Lowen2015} and the articles \cite{LiZ2018b,Windels}. \end{remark}
	
	Since $\sigma X$ is closed under formation of limits in $(\mathcal{P} X,\rho_X)$, for each $\phi\in\mathcal{P} X$ there is a smallest element  $c(\phi)$ in the complete lattice $(\sigma X,\sqsubseteq)$  that is larger than  $\phi$ (i.e., $\phi\sqsubseteq c(\phi)$ in $\mathcal{P} X$). We call $c(\phi)$ the \emph{closure} of $\phi$. The map $$c\colon(\mathcal{P} X,\rho_X)\longrightarrow(\sigma X,\rho_X)$$ is left adjoint to the inclusion map $$i\colon  (\sigma X,\rho_X)\longrightarrow (\mathcal{P} X,\rho_X).$$
	
	Let $f\colon(X,d_X)\longrightarrow(Y,d_Y)$ be a non-expansive map  preserving $\mathcal{T}$-colimits. Then $f$ induces an adjunction $$\hat{f}\dashv f^{-1}\colon (\sigma Y,\rho_Y)\longrightarrow (\sigma X,\rho_X). $$ The right adjoint $f^{-1}\colon (\sigma Y,\rho_Y)\longrightarrow (\sigma X,\rho_X)$ sends each $\mathcal{T}$-closed weight $\psi$ of $(Y,d_Y)$ to the $\mathcal{T}$-closed weight $\psi\circ f_*$ of $(X,d)$; the left adjoint  $\hat{f}\colon (\sigma X,\rho_X)\longrightarrow (\sigma Y,\rho_Y)$ sends each $\mathcal{T}$-closed weight $\phi$ of $(X,d)$ to the closure of the weight $ \phi\circ f^*$ of $(Y,d_Y)$; that is, $\hat{f}(\phi)=c(\phi\circ f^*)$.
	
	\begin{lemma} \label{s is T-cocontinuous} Let $\mathcal{T}$ be a saturated class of weights on $\sf qMet$. Then for each quasi-metric space $(X,d)$, the non-expansive map $${\sf s}\colon (X,d)\longrightarrow(\sigma X,\rho_X), \quad x\mapsto d(-,x)$$ preserves $\mathcal{T}$-colimits. \end{lemma}
	
	\begin{proof} Since $\sigma X$ is closed under formation of limits in $(\mathcal{P} X,\rho_X)$, the inclusion map $\sigma X\longrightarrow\mathcal{P} X$ has a left adjoint $c\colon \mathcal{P} X\longrightarrow\sigma X$. The map $c$ sends each $\phi\in\mathcal{P} X$ to its closure, i.e., the smallest $\mathcal{T}$-closed weight  larger than $\phi$ with respect to the uderlying order $\sqsubseteq$ of $(\sigma X,\rho_X)$. Since every representable weight is  $\mathcal{T}$-closed, the map $\sf s$ is the composite of the Yoneda embedding and $c$; that is, $${\sf s}=c\circ {\sf y}\colon X\longrightarrow\mathcal{P} X \longrightarrow\sigma X.$$ 
		
		Suppose that $\phi\in\mathcal{T} X$. We wish to show that $${\sf s}(\colim\phi)={\colim}_\phi{\sf s}$$ whenever $\phi$ has a colimit. 
		Since $c$ is a left adjoint, it preserves colimits, then  $${\colim}_\phi{\sf s}= {\colim}_\phi(c\circ{\sf y})= c({\colim}_\phi{\sf y})= c(\phi),$$ so it suffices to show that $c(\phi)=d(-,\colim\phi)$. On the one hand, since $d(-,\colim\phi)$ is $\mathcal{T}$-closed, then $c(\phi)\sqsubseteq d(-,\colim\phi)$. On the other hand, since $c(\phi)$ is $\mathcal{T}$-closed and $\phi\in\mathcal{T} X$, then  $c(\phi)(\colim\phi)\leq \rho_X(\phi,c(\phi))=0$, it follows that $d(-,\colim\phi)\sqsubseteq c(\phi)$. \end{proof}
	
	Let $$\kappa X$$ be the intersection of all subsets of the space $(\sigma X,\rho_X)$ that contain all representable weights of $(X,d)$ and are closed under formation of $\mathcal{T}$-colimits. It is not hard to see that $\kappa X$ is also closed under formation of $\mathcal{T}$-colimits,   it is the smallest subset of $(\sigma X,\rho_X)$  containing all representable weights and closed under formation of $\mathcal{T}$-colimits. In particular, $(\kappa X,\rho_X)$ is a $\mathbb{T}$-algebra.
	
	\begin{lemma} \label{t is T-cocontinuous}
		Let $\mathcal{T}$ be a saturated class of weights on $\sf qMet$. Then for each quasi-metric space $(X,d)$, the non-expansive map $${\sf t}\colon (X,d)\longrightarrow(\kappa X,\rho_X), \quad x\mapsto d(-,x)$$ preserves $\mathcal{T}$-colimits. 
	\end{lemma} 
	
	\begin{proof} This follows immediately from  the fact that ${\sf s}\colon (X,d)\longrightarrow(\sigma X,\rho_X)$ preserves $\mathcal{T}$-colimits and  
		$\kappa X$ is closed in $(\sigma X,\rho_X)$ under formation of $\mathcal{T}$-colimits. \end{proof}
	
	\begin{theorem}\label{idempotent completion} Let $\mathcal{T}$ be a saturated class of weights on $\sf qMet$. Then the category   of $\mathbb{T}$-algebras and homomorphisms is a reflective full subcategory of the category ${\sf qMet}^{\sf t}$ of quasi-metric spaces and non-expansive maps preserving $\mathcal{T}$-colimits.  
	\end{theorem}
	
	\begin{proof} By Lemma \ref{t is T-cocontinuous}, for each quasi-metric space $(X,d_X)$, the non-expansive map $${\sf t}\colon (X,d_X)\longrightarrow(\kappa X,\rho_X), \quad x\mapsto d_X(-,x)$$ preserves $\mathcal{T}$-colimits, hence a morphism of ${\sf qMet}^{\sf t}$. What remains is to show that for each $\mathcal{T}$-cocomplete quasi-metric space $(Y,d_Y)$ and each morphism $f\colon (X,d_X)\rightarrow (Y,d_Y)$ of ${\sf qMet}^{\sf t}$, there exists a unique homomorphism $\overline{f}\colon (\kappa X,\rho_X)\longrightarrow(Y,d_Y)$  of  $\mathbb{T}\text{-}{\sf Alg}$ such that $f=\overline{f}\circ {\sf t}$. $$\bfig\qtriangle<650,500>[(X,d_X)`(\kappa X,\rho_X)`(Y,d_Y); {\sf t}`f`\overline{f}] \efig$$
		
		We prove the existence first. Since $(Y,d_Y)$ is $\mathcal{T}$-cocomplete and  ${\sf s}\colon (Y,d_Y)\longrightarrow(\sigma Y,\rho_Y)$ preserves $\mathcal{T}$-colimits, it follows that the set $\{d_Y(-,y)\mid y\in Y\}$ of  representable weights is closed in $(\sigma Y,\rho_Y)$ under formation of $\mathcal{T}$-colimits, hence $\kappa Y= \{d_Y(-,y)\mid y\in Y\}$,   then $(Y,d_Y)\cong (\kappa Y,\rho_Y)$. Since $\hat{f}\colon (\sigma X,\rho_X)\longrightarrow (\sigma Y,\rho_Y)$ sends $d_X(-,x)$ to $d_Y(-,f(x))$ and preserves colimits,  it follows that the set $$\{\phi\in\sigma X\mid \hat{f}(\phi)\in \kappa Y\}$$ contains all  representable weights of $(X,d)$ and is closed under formation of $\mathcal{T}$-colimits. So, $\kappa X\subseteq \{\phi\in\sigma X\mid \hat{f}(\phi)\in \kappa Y\}$; that is,  $\hat{f}(\phi)\in \kappa Y$ for all $\phi\in\kappa X$. The  restriction of $\hat{f}$ on the subspace $(\kappa X,\rho_X)$ gives the desired homomorphism $\overline{f}\colon (\kappa X,\rho_X)\longrightarrow(Y,d_Y)$.
		
		For uniqueness, suppose that both $g_1\colon (\kappa X,\rho_X)\longrightarrow(Y,d_Y)$ and $g_2\colon (\kappa X,\rho_X)\longrightarrow(Y,d_Y)$ satisfy the requirement. Then the set $$ \{\phi\in\kappa X\mid g_1(\phi)=g_2(\phi)\}$$ contains all  representable weights of $(X,d)$  and is closed under formation of $\mathcal{T}$-colimits, hence equal to $\kappa X$, it follows that $g_1=g_2$.
	\end{proof}
	
	\begin{remark} \begin{enumerate} \item[\rm (i)] Theorem \ref{free cocompletion} is a special case of Kelly \cite[Theorem 5.35]{Kelly1982}, it holds for every saturated class of weights for quantale-enriched categories. This is also true for Theorem \ref{idempotent completion}; which is to say, Theorem \ref{idempotent completion} holds for every saturated class of weights for quantale-enriched categories. 
			\item[\rm (ii)] Since every class of weights has a (unique) saturation (see e.g. \cite{AK88,KS2005}), the assumption in Theorem \ref{idempotent completion} that the class of weights is saturated is not essential, it can be dropped. It should be noted that in the classical setting, i.e., in the case of the two-element quantale, Theorem \ref{idempotent completion} was proved in Zhao \cite[Theorem 12]{Zhao2015} for subset selections that are subset hereditary.
			\item[\rm (iii)] The reflector    ${\sf qMet}^{\sf t}\longrightarrow \mathbb{T}\text{-}{\sf Alg}$  sends each $\mathbb{T}$-algebra to itself, hence idempotent, but the free $\mathcal{T}$-cocompletion (i.e., the left adjoint of the forgetful functor  $\mathbb{T}\text{-}{\sf Alg}\longrightarrow {\sf qMet}$) is not idempotent in general. Applying Theorem \ref{idempotent completion} to the saturated class of ideals yields a construction of the idempotent Yoneda completion of quasi-metric spaces obtained in Ng and Ho \cite{NgH2017} via  formal balls. Applying it to the saturated class  of bounded ideals yields a construction of the idempotent local Yoneda completion of quasi-metric spaces obtained in Lu and Zhao \cite{Lu2023} via  formal balls.  \end{enumerate} \end{remark}
	
	The following conclusion extends Theorem 14 of Zhao \cite{Zhao2015} to the context of quasi-metric spaces. 
	
	\begin{corollary} Let $\mathcal{T}$ be a saturated class of weights on $\sf qMet$. Then for each quasi-metric space $(X,d)$, the quasi-metric space of $\mathcal{T}$-closed weights of $(X,d)$ is isometric to that of $\mathcal{T}$-closed weights of $(\kappa X,\rho_X)$.   \end{corollary}
	
	\begin{proof}   The non-expansive map ${\sf t}\colon (X,d)\longrightarrow(\kappa X,\rho_X), ~ {\sf t}(x)=d(-,x)$  preserves $\mathcal{T}$-colimits by Lemma \ref{t is T-cocontinuous}, so it induces an adjunction $$\hat{{\sf t}}\dashv {\sf t}^{-1}\colon (\sigma (\kappa X),\rho_{\kappa X})\longrightarrow (\sigma X,\rho_X). $$ The right adjoint ${\sf t}^{-1}$ maps each $\psi\in\sigma(\kappa X)$ to the $\mathcal{T}$-closed weight $\psi\circ{\sf t}_*$ of $(X,d)$,  the left adjoint $\hat{{\sf t}}$ maps  each $\phi\in \sigma X$  to the unique $\mathcal{T}$-closed weight $\overline{\phi}$ of  $(\kappa X,\rho_X)$ that makes the diagram
		$$\bfig\qtriangle<650,500>[(X,d)`(\kappa X,\rho_X)`({[0,\infty]},d_R); {\sf t}`\phi`\overline{\phi}] \efig$$    commute. So   $\hat{{\sf t}}$ and $ {\sf t}^{-1}$ are bijections, hence isometries. The conclusion thus follows.
	\end{proof}
	
	\section{Continuous algebras of the bounded ideal monad}
	
	\begin{definition}
		A  $\mathbb{J}$-algebra $(X,d)$ is continuous if the left adjoint $\colim\colon (\mathcal{J}X,\rho_X)\longrightarrow (X,d)$ of  $\mathsf{y}_X\colon (X,d)\longrightarrow (\mathcal{J}X,\rho_X)$ has a left adjoint. In other words, a quasi-metric space $(X,d)$ is a continuous $\mathbb{J}$-algebra if there exist  adjunctions 
		$$\wayb \dashv\colim \dashv\mathsf{y}_X\colon (X,d)\longrightarrow (\mathcal{J}X,\rho_X).$$ 
	\end{definition}
	
	Since the monad $\mathbb{J}=(\mathcal{J},{\sf m},{\sf y})$  is of Kock-Z\"{o}berlein type,  for each quasi-metric space $(X,d)$ we have a string of adjunctions $$\mathcal{J}\mathsf{y}_X  \dashv\mathsf{m}_X\dashv\mathsf{y}_{\mathcal{J}X}\colon  (\mathcal{J}X,\rho_X)\longrightarrow (\mathcal{JJ}X,\rho_{\mathcal{J}X}),$$ so $(\mathcal{J}X,\rho_X)$ is a continuous $\mathbb{J}$-algebra. 
	
	The aim of this section is to show that a quasi-metric space is a continuous  $\mathbb{J}$-algebra if and only if it is standard and its set of formal balls is a local domain. Recall that a poset  is a \emph{local domain} \cite{Mislove1999} if it is a continuous poset and a local dcpo.\footnote{In \cite{Mislove1999} a local domain is required to have a bottom element.}

	Let $(X,d)$ be a $\mathbb{J}$-algebra. The map $$\mathfrak{w}\colon X\times X\longrightarrow [0,\infty], \quad \mathfrak{w}(x,y)=\sup_{\phi\in \mathcal{J}X} (\phi(x)\ominus  d(y,\colim \phi)) $$   is  a distributor $(X,d)\oto(X,d)$, called the $\mathcal{J}$-below distributor of $(X,d)$  \cite{Waszkiewicz2009}.   
	It is not hard to check that a $\mathbb{J}$-algebra  $(X,d)$ is continuous if and only if for all $x\in X$, the weight $\mathfrak{w}(-,x)$ belongs to $\mathcal{J}X$ and has $x$ as a colimit.  In this case,  $$\wayb x=\mathfrak{w}(-,x).$$
	Furthermore,   the $\mathcal{J}$-below distributor $\mathfrak{w}$ of a continuous $\mathbb{J}$-algebra 
	is interpolative in the sense that $\mathfrak{w}\circ \mathfrak{w}=\mathfrak{w}$. See e.g. \cite{HW2011,Waszkiewicz2009}.

	\begin{theorem}\label{local domain vs continuous alg} Let $(X,d)$ be a quasi-metric space. 
		Then the following are equivalent:  \begin{enumerate}
			\item[\rm(1)] $(X,d)$ is a continuous  $\mathbb{J}$-algebra. \item[\rm(2)] $(X,d)$ is standard and $(\mathrm{B}X,\sqsubseteq)$ is a local domain.
		\end{enumerate}   In this case, the $\mathcal{J}$-below distributor $\mathfrak{w}$ of $(X,d)$ and the way below relation $\ll$ of   $(\mathrm{B}X,\sqsubseteq)$ determine each other via $$(x,r)\ll(y,s)\quad\text{if and only if}\quad  r>s+\mathfrak{w}(x,y)$$ for all $x,y\in X$ and $r,s<\infty$.
	\end{theorem}

	\begin{proof} $(1)\Rightarrow(2)$ It suffices to show that $(\mathrm{B}X,\sqsubseteq)$ is a continuous poset. For each $y\in X$, let $$d(y)=\{(x,r)\in \mathrm{B}X\mid  \mathfrak{w}(x,y)=\wayb y(x)<r \}.$$ 
		Since $\wayb y$ is an ideal with colimit $y$, it follows that $d(y)$ is a directed set of $\mathrm{B}X$ with $(y,0)$ as a join.  We assert that  $(x,r)\ll (y,0)$ for all $(x,r)\in d(y)$,  and  that   $(x,r)\ll (y,s)$ whenever $r>\mathfrak{w}(x,y)+s$, in particular, $(x,s+r)\ll (y,s)$ for all $(x,r)\in d(y)$.
		From these assertions one easily deduces that $(\mathrm{B}X,\sqsubseteq)$ is a continuous poset. The proof of the  assertions is similar to those in \cite[Theorem 17.16]{Zhang2024}, so we omit it here.

		$(2)\Rightarrow(1)$ For each $y\in X$, let 
		$$D=\{(x,r)\in \mathrm{B}X\mid (x,r)\ll (y,0)\}.$$ 
		Then $D$ is a directed set of $(\mathrm{B}X,\sqsubseteq)$ with $(y,0)$ as a join. Index $D$ by itself as $D=\{(x_i,r_i)\}_{i\in D}$, which is to say, $(x_i,r_i)\sqsubseteq(x_j,r_j)$ if and only if $i\leq j$. Since $(X,d)$ is standard, it follows that $y$ is a Yoneda limit of the forward Cauchy net $\{x_i\}_{i\in D}$, hence  a colimit of the ideal $$k(y)=\inf_{i\in D} \sup_{j\geq i} d(-,x_j).$$  We assert that   $k(y)(x)<r$ implies $(x,r)\ll (y,0)$, and the latter implies $(x,r+t)\ll (y,t)$ for all $t<\infty$.   With help of this assertion one readily verifies that the assignment $y\mapsto k(y)$ defines a left adjoint of $\colim \colon(\mathcal{J}X,\rho_X)\longrightarrow (X,d)$, hence $(X,d)$ is a continuous $\mathbb{J}$-algebra. Verification of the  assertion  is left to the reader (see \cite[Theorem 17.16]{Zhang2024}).
	\end{proof}	 
	
	\section*{Acknowledgement}
	This paper is completed during the first author's  visit to Sichuan University, September 2023 to August 2024. The first author  acknowledges the support of the National Natural Science Foundation of China (No. 12001473), the Shandong Provincial Natural Science Foundation of China (ZR2023MA059) and the Visiting Research Funds of Shandong Provincial Universities. The second author  acknowledges the support of the National Natural Science Foundation of China (No. 12371463).

\end{document}